\documentclass[10pt]{amsart}
\usepackage[utf8]{inputenc}
\usepackage[ngerman, english]{babel}
\usepackage[utf8]{inputenc}
\usepackage{pdfsync}
\usepackage{verbatim}
\usepackage[monochrome]{color}
\usepackage{amsmath}
\usepackage{amsthm}
\usepackage{amssymb}
 \usepackage{paralist}
\newtheorem{thm}{Theorem}[section]

  \newtheorem{lemma}[thm]{Lemma}
  \newtheorem{cor}[thm]{Corollary}
  \newtheorem{thmx}{Theorem}
   
  \newtheorem{corx}[thmx]{Corollary}
  \newtheorem{prop}[thm]{Proposition}
  \newtheorem{quest}[thm]{Question}
  
  \theoremstyle{definition}
 \newtheorem{defi}[thm]{Definition}
  \newtheorem{bem}[thm]{Remark}
    {\begin{proof}[Beweis]}
    {\end{proof}}
  \newtheorem{ex}[thm]{Example}

\newcommand{\norm}[1]{\lVert#1\rVert}   
\newcommand{\abs}[1]{\lvert#1\rvert}

\newcommand{\ZZ}{\mathbb Z}

\usepackage{tikz}
\usetikzlibrary{matrix,arrows}
\usepackage{hyperref}
\allowdisplaybreaks
\title[Nuclear dimension of subhomogeneous twisted groupoid C*-algebras]{Nuclear dimension of subhomogeneous twisted groupoid C*-algebras and dynamic asymptotic dimension}
\author[C.\ B\"onicke]{Christian B\"onicke}
\address{School of Mathematics, Statistics and Physics, Newcastle University, Newcastle upon Tyne NE1 7RU, United Kingdom}
\email{christian.bonicke@ncl.ac.uk}
\author[K.\ Li]{Kang Li}
\address{Department of Mathematics, Friedrich-Alexander-Universität Erlangen-Nürnberg,  Cauerstraße 11, 91058 Erlangen, Germany}
\email{kang.li@fau.de}
\date{\today}

\usepackage{graphicx}

\begin{document}

\begin{abstract}
We characterise subhomogeneity for twisted \'etale groupoid $\mathrm{C}^*$-algebras and obtain an upper bound on their nuclear dimension. As an application, we remove the principality assumption in recent results on upper bounds on the nuclear dimension of a twisted \'etale groupoid $\mathrm{C}^*$-algebra in terms of the dynamic asymptotic dimension of the groupoid and the covering dimension of its unit space. As a non-principal example, we show that the dynamic asymptotic dimension of any minimal (not necessarily free) action of the infinite dihedral group $D_\infty$ on an infinite compact Hausdorff space $X$ is always one. So if we further assume that $X$ is second-countable and has finite covering dimension, then $C(X)\rtimes_r D_\infty$ has finite nuclear dimension and is classifiable by its Elliott invariant.
\end{abstract}

\maketitle

\section{Introduction}
The nuclear dimension for $\mathrm{C}^*$-algebras is a non-commutative version of Lebesgue covering dimension introduced by Winter and Zacharias \cite{WZ10}. It played a crucial role in the Elliott classification programme for simple nuclear $\mathrm{C}^*$-algebras. More precisely, all non-elementary separable simple  $\mathrm{C}^*$-algebras with finite nuclear dimension satisfying the Universal Coefficient Theorem are classifiable by their Elliott invariant (see \cite{MR3583354, MR4215379, MR4215380, EGLN}). In the sequel, we refer to this class as “classifiable $\mathrm{C}^*$-algebras”.


In this article, we focus on the nuclear dimension of $\mathrm{C}^*$-algebras arising from twisted \'etale groupoids. This class of $\mathrm{C}^*$-algebras is on the one hand large enough to cover many examples of interest. In fact, all classifiable $\mathrm{C}^*$-algebras admit a twisted \'etale groupoid model \cite{L20}. More generally, every $\mathrm{C}^*$-algebra admitting a Cartan subalgebra has a twisted \'etale groupoid model \cite{R08,MR4489313}. On the other hand, the nuclear dimension of twisted \'etale groupoid $\mathrm{C}^*$-algebras is naturally related to the dimensions of the underlying groupoids. More specifically, Guentner, Willett, and Yu introduced in \cite{GWY17} the concept of the dynamic asymptotic dimension of an \'etale groupoid $G$ (denoted $\mathrm{dad}(G)$), and proved that the nuclear dimension of the reduced groupoid $\mathrm{C}^*$-algebra $C_r^*(G)$ is bounded above by a number depending on the dynamic asymptotic dimension of $G$ and the covering dimension of its unit space $G^0$ provided that $G$ is principal (i.e., when the isotropy groups $G_x^x=\{g\in G\mid s(g)=x=r(g)\}$ are trivial for all $x\in G^0$). Recently, this result has been generalized to \emph{twisted} \'etale groupoids \cite{CDGHV22} still under the principality assumption.

The strategy of the proofs for the main results in \cite{GWY17,CDGHV22} is as follows: when $G$ has finite dynamic asymptotic dimension $d$ and $\Sigma$ is any twist over $G$, then $C_r^*(G;\Sigma)$ can be locally approximated by $d+1$ sub-$C^*$-algebras associated with certain open precompact subgroupoids $H_i$ of $G$, $0\leq i\leq d$. This reduces the problem of bounding the nuclear dimension of $C_r^*(G;\Sigma)$ to bounding the nuclear dimension of each of these $d+1$ sub-$\mathrm{C}^*$-algebras associated with the groupoids $H_i$. As each $H_i$ is an open precompact subgroupoid of $G$, it turns out that these sub-$\mathrm{C}^*$-algebras are all subhomogeneous\footnote{Recall that a $C^*$-algebra $A$ is \emph{subhomogeneous} if there is a finite upper bound on the dimension of the irreducible representations of $A$.}. Luckily, the nuclear dimension of subhomogeneous $\mathrm{C}^*$-algebras has already been analysed and computed by Winter in \cite{W04}.

In our main result, we characterise subhomogeneity for twisted \'etale groupoid $\mathrm{C}^*$-algebras and then use Winter's result in \cite{W04} to give an estimate on their nuclear dimension (see Proposition~\ref{prop:sh}, Theorem~\ref{Thm:DimShmg} and Proposition~\ref{nucdim virab}):
\begin{thmx}\label{Thm:A}
    Let $G$ be a locally compact, second-countable, Hausdorff \'etale groupoid and let $\Sigma$ be a twist over $G$. {\color{red} For each $x\in G^0$ let $\sigma_x:G_x^x\times G_x^x\to \mathbb T$ denote the $2$-cocycle associated with the restriction of the twist $\Sigma$ to $x$.} Then $C^*(G;\Sigma)$ is subhomogeneous if and only if
    $$\sup_{x\in G^0}\sup_{\pi\in (\widehat{G_x^x,\sigma_x)}}\abs{G_x/G_x^x}\cdot \dim(H_\pi)<\infty.$$
    In this case, the nuclear dimension of $C^*(G;\Sigma)$ can be bounded above as follows:
    $$
    \dim_{\mathrm{nuc}}^{+1}(C^*(G;\Sigma))\leq \dim^{+1}(G^0)\cdot \sup_{x\in G^0} \mathrm{asdim}^{+1}(G_x^x),
    $$
    {\color{red} where $\mathrm{asdim}(G_x^x)$ denotes Gromov's asymptotic dimension of the countable group $G_x^x$.}
\end{thmx}
Replacing \cite[Proposition~4.3]{CDGHV22} about the nuclear dimension of subhomogenous \emph{principal} twisted \'etale groupoid $\mathrm{C}^*$-algebras in the argument outlined above with our Theorem \ref{Thm:A}, allows us to drop the principality assumption from the main results of \cite{GWY17,CDGHV22} (see Theorem~\ref{nucdim}):

\begin{thmx}\label{Thm:B}
    Let $G$ be a second-countable, locally compact, Hausdorff \'etale groupoid and let $\Sigma$ be a twist over $G$. Then
    $$\dim_{\mathrm{nuc}}^{+1}(C_r^*(G;\Sigma))\leq\mathrm{dad}^{+1}(G)\cdot \dim^{+1}(G^0).$$
\end{thmx}

Combining \cite[Theorem~1.1]{MR3631229} and \cite[Theorem~4.7]{GWY17} with Theorem~\ref{Thm:B}, we obtain the following corollary: 
{\color{red}\begin{corx}
    Let $\Gamma$ be a finitely generated virtually nilpotent group acting on a compact metric space of finite covering dimension. If all stabilizer subgroups are finite, then both the dynamic asymptotic dimension and the nuclear dimension are finite\footnote{If we remove the condition “all stabilizer subgroups are finite”, then the dynamic asymptotic dimension is often infinite but the nuclear dimension is always finite by \cite[Corollary~10.6]{HW23}.}.
\end{corx}}

Note that having finite dynamic asymptotic dimension forces all the isotropy
groups $G_x^x$ to be locally finite. Nevertheless, there are many interesting examples with finite isotropy
groups such as minimal actions of the infinite dihedral group (see Theorem~\ref{dad D_inf} and Example~\ref{ex:Edu}): 
\begin{thmx}\label{Thm:C}
The dynamic asymptotic dimension of any minimal action $D_\infty \curvearrowright X$ of the infinite dihedral group $D_\infty$ on an infinite compact Hausdorff space $X$ is one.

\smallskip

If we further assume that $X$ is second-countable and has finite covering dimension, then $C(X)\rtimes D_\infty$ is classifiable
by its Elliott invariant and has nuclear dimension at most one.
\end{thmx}

Finally, as an application we obtain the following corollary (see Corollary~\ref{vircyc nuc}):
\begin{corx}\label{Cor:D}
   Let $X$ be an infinite compact Hausdorff space. If $\Gamma$ is a virtually cyclic group acting minimally on $X$, then the dynamic asymptotic dimension of $\Gamma  \curvearrowright X$ is one and
\begin{align*}
\dim_{\mathrm{nuc}}(C(X)\rtimes_r \Gamma)\leq 2\cdot \dim(X)+1.
\end{align*}
\end{corx}
The article is organised as follows: In section \ref{Sec:Subhom} we present our characterisation of subhomogeneity for twisted \'etale groupoid $\mathrm{C}^*$-algebras and prove Theorem \ref{Thm:A}. In Section \ref{Sec:dad} we discuss several applications, including the proofs of Theorem~\ref{Thm:B}, Theorem~\ref{Thm:C} and Corollary~\ref{Cor:D}.
\section{Subhomogeneous twisted groupoid C*-algebras}
\label{Sec:Subhom}

\textbf{Throughout the article $G$ will denote an \'etale, locally compact Hausdorff groupoid}. We will denote its space of units by $G^0$ and the source and range maps by $s,r\colon G \to G^0$, respectively. We will use the standard notation $G_x:=\{g\in G\mid s(g)=x\}$, $G^x:=\{g\in G\mid r(g)=x\}$ and $G_x^x:=G_x\cap G^x$ for the range fibre, the source fibre, and the isotropy group at $x\in G^0$, respectively. Recall, that a \emph{twist} over $G$ is a central groupoid extension 
$$G^0\times \mathbb T\stackrel{i}{\rightarrow} \Sigma\stackrel{j}{\rightarrow} G,$$
meaning that
\begin{enumerate}
    \item $i$ is a homeomorphism onto $j^{-1}(G^0)$,
    \item $j$ is a continuous and open surjection, and
    \item the extension is central in the sense that $i(r(\sigma),z)\sigma=\sigma i(s(\sigma),z)$.
\end{enumerate}

To construct the twisted groupoid $\mathrm{C}^*$-algebra one forms the associated line bundle $L_\Sigma=\mathbb C\times \Sigma/{\sim}\rightarrow G$ and considers the set $\Gamma_c(G,L_\Sigma)$ of continuous, compactly supported sections of this bundle. Equipped with a suitably defined involution and convolution it becomes a $\ast$-algebra that can be completed with respect to suitable norms to form the full and reduced twisted groupoid $\mathrm{C}^*$-algebras $C^*(G;\Sigma)$ and $C_r^*(G;\Sigma)$, respectively (see for example \cite{R08} for the details).

For every $x\in G^0$ the twist $\Sigma$ over $G$ restricts to a twist $\mathbb T\rightarrow \Sigma_x^x\rightarrow G^x_x$ over the discrete group $G_x^x$. Such a twist gives rise to a $2$-cocycle $\sigma_x$ on $G_x^x$. In fact a bit more is true, so let us review the construction.

For each $x\in G^0$ let $\theta_x:G_x\rightarrow \Sigma_x$ be a section of the restriction of $j$ to $\Sigma_x$ such that $\theta_x(x)=x$. Then $j(\theta_{s(h)}(gh)^{-1}\theta_{s(g)}(g)\theta_{s(h)}(h))=(gh)^{-1}gh=s(h)$ for all $(g,h)\in G^{(2)}$. It follows that $\theta_{s(h)}(gh)^{-1}\theta_{s(g)}(g)\theta_{s(h)}(h)$ actually lives in the image of $\mathbb T\cong \{s(h)\}\times \mathbb T$ inside $\Sigma_{s(h)}$. So identifying $\mathbb T$ with its image in $\Sigma_{x}$ allows us to define a maps
$$\sigma_x:G_x\times G^x\rightarrow \mathbb T$$
by $\sigma_x(g,h)=\theta_{s(h)}(gh)^{-1}\theta_x(g)\theta_{s(h)}(h)$.
One can easily check that $\sigma_x$ satisfies the cocycle identity
$$\sigma_x(g_1,g_2)\sigma_y(g_1g_2,g_3)=\sigma_x(g_1,g_2g_3)\sigma_y(g_2,g_3)\ \forall g_1\in G_x, g_2\in G_y^x, g_3\in G^y$$
Since $\theta_x(x)=x$ it is also routine to verify that $\sigma_x$ is normalised in the sense that $\sigma_x(g,x)=\sigma_x(x,h)=1$ for all $g\in G_x$ and $h\in G^x$.

\subsection{A characterisation of subhomogeneity}
Recall, that a $\sigma_x$-representation of $G_x^x$ is a map
$\pi:G_x^x\rightarrow B(H_\pi)$ such that $\sigma_x(g,h)\pi(gh)=\pi(g)\pi(h)$. The set of (equivalence classes of) irreducible $\sigma_x$-representations will be denoted by $\widehat{(G_x^x,\sigma_x)}$
Central to our analysis is the procedure of inducing $\sigma_x$-representations of an isotropy group $G_x^x$ to a representation of $C^*(G;\Sigma)$. There are two approaches to this which we shall now describe. 

{\color{red} The first of these is more concrete and inspired by Mackey's classical theory of induced representations for groups (compare \cite[Chapter~8]{CSST22}, or \cite[Section~3]{CaH12} for a related construction in the groupoid case). Since none of the references describe exactly what we need, we spell out the constructions but omit the proofs as they are all classical.} Consider the complex vector space of all functions $\xi:G_x\rightarrow H_\pi$ such that $\xi(gh)=\overline{\sigma_x(g,h)}\pi_h^*\xi(g)$ for all $g\in G_x$ and $h\in G_x^x$ and such that $\sum_{gG_x^x\in G_x/G_x^x} \norm{\xi(g)}^2<\infty$.
There is an inner product on this vector space given by
$$\langle \xi,\eta\rangle:= \sum_{gG_x^x\in G_x/G_x^x} \langle \xi(g),\eta(g)\rangle.$$
Note that this is well-defined since 
$$\langle \xi(gh),\eta(gh)\rangle=\vert\sigma_x(g,h)\vert^2 \langle \pi_h^*\xi(g),\pi_h^*\eta(g)\rangle=\langle \xi(g),\eta(g)\rangle.$$
The completion with respect to this inner product is a Hilbert space and will be denoted $\mathrm{Ind}\ H_\pi$ in the sequel.
On this Hilbert space one can define the \emph{induced representation of $\pi$}
$$\mathrm{Ind}_{G_x^x}^{G}\ \pi:C^*(G;\Sigma)\rightarrow B(\mathrm{Ind}\ H_\pi)$$
for $f\in \Gamma_c(G,L_\Sigma)$, $\xi:G_x\rightarrow H_\pi$ and $g\in G_x$ by 
$$(\mathrm{Ind}_{G_x^x}^{G}\ \pi(f)\xi)(g)=\sum_{h\in G^{r(g)}} \overline{\sigma_{s(h)}(h,h^{-1}g)}f(h)\xi(h^{-1}g).$$ 

One of the advantages of this concrete approach is that it allows for a simple proof of the following result:
\begin{lemma}\label{Lemma:DimInd} Let $\pi:G_x^x\rightarrow B(H_\pi)$ be a $\sigma_x$-representation. Then the Hilbert space $\mathrm{Ind}\ H_\pi$ for the induced representation is finite-dimensional if and only if $H_\pi$ is finite dimensional and $G_x/G_x^x$ is finite. In this case we have
    $$\dim (\mathrm{Ind}\ H_\pi)=\vert G_x/G_x^x\vert \dim(H_\pi).$$
\end{lemma}
\begin{proof}
     Let $(\xi_i)_{i\in I}$ be an orthonormal basis for $H_\pi$ and let $(g_j)_{j\in J}$ be a complete system of representatives of the orbit space $G_x/G_x^x$. 
    For each $i\in I$ and $j\in J$ define a function $\varphi_{i,j}:G_x\rightarrow H_\pi$ by setting $$\varphi_{i,j}(g)=\left\{\begin{array}{cl}
       \overline{\sigma_x(g_j,h)}\pi_h^* \xi_i  &  g=g_jh \textit{ for some }h\in G_x^x\\
       0  & \textit{otherwise}
    \end{array}\right\}.$$
Suppose $g\in G_x$ and $h\in G_x^x$. Then there exists a unique $j\in J$ such that $g=g_jh'$, so we get $gh=g_jh'h$. Hence we can use the cocycle identity to compute
\begin{align*}
    \varphi_{i,j}(gh) &= \overline{\sigma_x(g_j,h'h)}\pi_{h'h}^*\xi_i\\
    & =\overline{\sigma_x(g_j,h')\sigma_x(g_jh',h)}\sigma_x(h',h)\pi_{h^{-1}h'^{-1}}\xi_i \\
    & = \overline{\sigma_x(g_j,h')\sigma_x(g_jh',h)}\sigma_x(h^{-1},h'^{-1})\pi_{h^{-1}h'^{-1}}\xi_i \\
    & = \overline{\sigma_x(g_jh',h)\sigma_x(g_j,h')}\pi_{h}^*\pi_{h'}^*\xi_i \\
    & = \overline{\sigma_x(g,h)}\pi_h^* \varphi_{i,j}(g)
\end{align*}

It follows, that $\varphi_{i,j}\in \mathrm{Ind} \ H_\pi$, and in fact, $(\varphi_{i,j})_{i\in I, j\in J}$ is easily seen to be an orthonormal family in $\mathrm{Ind}\ H_\pi$.

Conversely, if $\varphi:G_x\rightarrow H_\pi$ is any function as in the definition of $\mathrm{Ind}\ H_\pi$, write $\varphi(g_j)=\sum_{j\in J} \lambda_{i,j} \xi_i$. Then $\varphi=\sum_{i\in I}\sum_{j\in J} \lambda_{i,j}\varphi_{i,j}$. It follows that $(\varphi_{i,j})_{i,j}$ is an orthonormal basis, and the claims in the statement follow.
\end{proof}

An alternative route to obtain the induced representation is via Hilbert modules. {\color{red} To describe it, recall the definition of the maximal twisted group $\mathrm{C}^*$-algebra $C^*(G_x^x;\sigma_x)$: consider the complex vector space $\ell^1(G_x^x;\sigma_x)$ of all summable complex functions on $G_x^x$ equipped with the twisted convolution and involution given by
$$
f_1*f_2(g)=\sum_{h\in G_x^x}\sigma_x(h,h^{-1}g)f_1(h)f_2(h^{-1}g)
$$
and $$
f^*(g)=\overline{\sigma_x(g,g^{-1})}\overline{f(g^{-1})}.
$$
Then $C^*(G_x^x;\sigma_x)$ is the enveloping $\mathrm{C}^*$-algebra of the involutive Banach algebra $\ell^1(G_x^x;\sigma_x)$.
Equivalently, $C^*(G_x^x;\sigma_x)$ is the universal $\mathrm{C}^*$-algebra for $\sigma_x$-representations of $G_x^x$.

The following description of induction via Hilbert modules is taken from \cite[Section~4.1]{SW13} applied to the special case of line bundles and the subgroupoid $G_x^x$.} Consider the complex vector space $C_c(G_x)$ and define a $C^*(G_x^x;\sigma_x)$-valued inner product for $\xi,\eta\in C_c(G_x)$ by
    $$
    \langle \xi,\eta\rangle:=\xi^* \ast_{\sigma_x} \eta,
    $$
    where the latter denotes the twisted convolution.
Completing $C_c(G_x)$ with respect to this inner product one obtains a Hilbert $C^*(G_x^x;\sigma_x)$-module $\mathsf{X}$ on which $C^*(G;\Sigma)$ acts from the left by adjointable operators.
To induce a $\sigma_x$-representation $\pi$ of $G_x^x$ one first considers the integrated form $\Tilde{\pi}:C^*(G_x^x;\sigma_x)\rightarrow B(H_\pi)$ of $\pi$
and then forms the balanced tensor product $\mathsf{X}\otimes_{\tilde{\pi}} H_\pi$. 
The latter is then a Hilbert space on which one can define the \textit{induced representation}
$$\mathrm{Ind}_\mathsf{X} \tilde{\pi}:C^*(G;\Sigma)\rightarrow B(\mathsf{X}\otimes_{\tilde{\pi}} H_\pi )$$
by the formula
$$\mathrm{Ind}_\mathsf{X} \tilde{\pi}(f)(\varphi\otimes h)=(f\cdot\varphi)\otimes h.$$

The following result is well-known in the untwisted setting and it is routine to extend it to the twisted one. It shows that the two approaches to induction explained above yield the same result up to unitary equivalence:
\begin{lemma}\label{Lemma:Mackey Induction and elementary induction are the same}
Let $\pi:G_x^x\rightarrow B(H_\pi)$ be a $\sigma_x$-representation. Then there exists a unitary operator $U:\mathsf{X}\otimes_{\tilde{\pi}} H_\pi\rightarrow \mathrm{Ind}\ H_\pi$ that intertwines $\mathrm{Ind}_\mathsf{X} \tilde{\pi}$ and $\mathrm{Ind}_{G_x^x}^{G}\ \pi$.
\end{lemma}

Induction via the Hilbert module $\mathsf{X}$ is the preferred method to prove the following result, which is (a special case of) the main result of \cite[Theorem~6.3]{IW15} (we only use the case of line bundles).
\begin{thm}
Let $x\in G^0$ and $\pi\in \widehat{(G_x^x;\sigma_x)}$ be an irreducible representation. Then $\mathrm{Ind}_\mathsf{X} \tilde{\pi}$ is an irreducible representation of $C^*(G;\Sigma)$.
\end{thm}

Thus, twisted irreducible representations of the isotropy groups give rise to irreducible representations of $C^*(G;\Sigma)$.
Before we turn to our characterisation of subhomogeneity, we prove some sufficient conditions for a twisted groupoid $\mathrm{C}^*$-algebra to be GCR or CCR.\footnote{A $\mathrm{C}^*$-algebra $A$ is said to be a CCR-algebra (respectively, a GCR-algebra) if the relation $\pi(A)=K(H_\pi)$ (respectively, $\pi(A)\supseteq K(H_\pi)$) is satisfied for any non-trivial irreducible representation $\pi$ of $A$ in a Hilbert space $H_\pi$.}
In the untwisted setting Clark gave a nice characterisation for $C^*(G)$ to be GCR (resp. CCR) in {\color{red} \cite[Theorems~1.3 and 1.4]{C07} and this was later improved by van Wyk in \cite[Theorem~4.3]{vW18} and \cite[Theorem~3.6]{vW19}. A full analogue of this result is not available in the literature for the twisted case. There are partial results in \cite[Proposition~3.3]{CaH12} that deal with the case of twists over a principal groupoid, but our main interest is really beyond the principal case.}
However, for twists over \'etale groupoids we can exploit an algebraic characterisation of GCR/CCR twisted group $C^*$-algebras to at least obtain sufficient conditions for $C^*(G;\Sigma)$ to be GCR or CCR as follows:
\begin{prop}\label{Prop:GCR/CCR}
Let $G$ be a second-countable, locally compact, Hausdorff étale groupoid and $\Sigma$ be a twist over $G$. Then the following hold:
\begin{enumerate}
\item If $G^0/G$ is $T_0$ and $C^*(G_x^x,\sigma_x)$ is GCR for all $x\in G^0$, then $C^*(G;\Sigma)$ is GCR.
\item If $G^0/G$ is $T_1$ and $C^*(G_x^x,\sigma_x)$ is CCR for all $x\in G^0$, then $C^*(G;\Sigma)$ is CCR.    
\end{enumerate}    
\end{prop}
\begin{proof}
We only prove the first item as the second one can be proven in exactly the same way by replacing GCR by CCR and $T_0$ by $T_1$ throughout the proof.
{\color{red}
By \cite[Theorem~1.1]{H81} the twisted group $(G_x^x,\sigma_x)$ is GCR, or equivalently, type I if and only if there exists an abelian subgroup $\Lambda_x\leqslant G_x^x$ of finite index on which $\sigma_x$ is symmetric.} It is then clear that the $2$-cocycle $\sigma_x^n$ given by $\sigma_x^n(g,h):=\sigma_x(g,h)^n$ is also symmetric on $\Lambda_x$ for all $n\in\ZZ$ and hence $(G_x^x,\sigma_x^n)$ is type I for all $n\in \mathbb{Z}$ {\color{red} by invoking \cite[Theorem~1.1]{H81} again}. It follows that $C^*(\Sigma_x^x)\cong \bigoplus_{n\in\mathbb Z} C^*(G_x^x,\sigma_x^n)$ is type I for all $x\in G^0$. Moreover, $\Sigma^0/\Sigma\cong G^0/G$ is $T_0$ and hence \cite[Theorem~1.4]{C07} implies that $C^*(\Sigma)$ is GCR. 
  By \cite[Proposition~3.7]{IKRSW21} we have $C^*(\Sigma)\cong \bigoplus_{n\in\mathbb Z} C^*(G;\Sigma_n)$ and since quotients of GCR-algebras are GCR, we get that $C^*(G;\Sigma)$ is GCR.
\end{proof}
The converse of this result is not clear to us as there is no longer an obvious continuous inclusion $G^0/G \rightarrow \widehat{C^*(G;\Sigma)}$ that allows to pull back topological information from the spectrum to the orbit space.






We are now ready for the main result of this section:
\begin{prop}\label{prop:sh} Let $G$ be a second-countable, locally compact, Hausdorff étale groupoid and $\Sigma$ be a twist over $G$. Then $C^*(G;\Sigma)$ is subhomogeneous if and only if the following hold:
\begin{enumerate}
    \item $\sup_{x\in G^{(0)}}\abs{G_x/G_x^x}<\infty$, and
    \item $\sup_{x\in G^{(0)}}\sup_{\pi\in \widehat{(G_x^x,\sigma_x)}} \dim(H_\pi)<\infty.$
\end{enumerate}
\end{prop}
\begin{proof}
    Suppose first that $C^*(G;\Sigma)$ is subhomogeneous. Then there exists a uniform upper bound $N\in \mathbb N$ for the dimensions of irreducible representations of $C^*(G;\Sigma)$. Let $x\in G^0$ and $\pi$ be a twisted irreducible representation of $G_x^x$. Then $\mathrm{Ind}_{G_x^x}^G \pi$ is an irreducible representation of $C^*(G;\Sigma)$ and Lemma \ref{Lemma:DimInd} implies $$\vert G_x/G_x^x\vert\dim(H_\pi)\leq\dim (\mathrm{Ind}_{G_x^x}^G \pi)\leq N.$$
    Since $x$ and $\pi$ were arbitrary, this implies conditions (1) and (2).

    Conversely, we know that for every $x\in G^0$ there is a bijection between the orbit $Gx$ and the quotient space $G_x/G_x^x$. Using this, condition (1) implies that every orbit is a finite subset of $G^0$. Since $G^0$ is Hausdorff this implies that every orbit is closed. It follows in particular that $G^0/G$ is $T_1$. Moreover, the second condition implies that $C^*(G_x^x;\sigma_x)$ is subhomogeneous and so in particular CCR for every $x\in G^0$. Consequently, $C^*(G;\Sigma)$ is CCR by Proposition \ref{Prop:GCR/CCR}.
    We need to find an upper bound on the dimensions of irreducible representations of $C^*(G;\Sigma)$. Since the orbit space of $G$ is $T_1$ we know that every irreducible representation of $C^*(G;\Sigma)$ is is induced from an isotropy group, i.e. it is of the form $\mathrm{Ind}_{G_{x}^{x}}^G \pi$ for some irreducible $\sigma_x$-representation $\pi$ of $G_x^x$. Again, it can be interpreted as an irreducible representation on the Hilbert space $\mathrm{Ind}\ H_\pi$. By $(2)$ we have $\dim(H_\pi)<\infty$ and hence another application of Lemma \ref{Lemma:DimInd} implies that $\mathrm{Ind}\ H_\pi$ is finite dimensional if and only if the quotient space $G_x/G_x^x$ is finite and in that case $\dim(\mathrm{Ind}\ H_\pi)=\abs{G_x/G_x^x}\dim(H_\pi)$. Thus, the result follows.
\end{proof}

\subsection{Nuclear dimension of subhomogeneous twisted groupoid $C^*$-algebras}
Now that we have a satisfactory description of subhomogeneous groupoid $C^*$-algebras we want to estimate their nuclear dimension. The main tool to achieve this is the following result of Winter.
\begin{thm}[\cite{W04}] \label{Thm:Winter}
Let $A$ be a separable, subhomogeneous $\mathrm{C}^*$-algebra. Then we have
$$\mathrm{dim}_{\mathrm{nuc}}(A)=\max_{k\in \mathbb{N}}\lbrace \dim \mathrm{Prim}_k(A)\rbrace.$$
\end{thm}

To use Winter's result effectively, we first need a better understanding of the topology of $\mathrm{Prim}(C^*(G;\Sigma))$. 
Recall, that the quasi-orbit space $\mathcal{Q}(G)$ of $G$ is the quotient of $G^0$ by the equivalence relation that identifies two points $x,y\in G^0$ if their orbit closures agree, i.e. $\overline{Gx}=\overline{Gy}$.

As a first step we will see that the map $p:\mathrm{Prim}(C^*(G;\Sigma))\rightarrow \mathcal{Q}(G)$ that associates to every kernel of an irreducible representation the closure of the orbit it lives on is continuous. To do this we need to make our colloquial description of $p$ a bit more precise. The key ingredient is the following Lemma:
\begin{lemma}
Let $G$ be an \'etale groupoid and $\Sigma$ a twist over $G$. Then there exists a canonical non-degenerate homomorphism
$$M:C_0(G^0)\rightarrow \mathcal{M}(C^*(G;\Sigma)).$$
\end{lemma}
\begin{proof}
For $\varphi\in C_0(G^0)$, $f\in C_c(G;\Sigma)$, and $\sigma\in \Sigma$ we define 
$$(M_\varphi f)(\sigma)=\varphi(r(\sigma))f(\sigma),\text{ and}$$
$$(fM_\varphi)(\sigma)=\varphi(s(\sigma))f(\sigma).$$
Then $M_\varphi$ acts as a double centralizer on $C_c(G;\Sigma)$ and hence extends to an element in $\mathcal{M}(C^*(G;\Sigma))$.
It is easy to see that this is a $\ast$-homomorphism. It is non-degenerate since $C_0(G^0)$ contains an approximate unit for $C^*(G;\Sigma)$.
\end{proof}

Following the discussion on page 61 in \cite{RW98}, $M$ induces a continuous restriction map
$$\mathrm{Res}_M:\mathcal{I}(C^*(G;\Sigma))\rightarrow \mathcal{I}(C_0(G^0)),$$
between the lattices of closed two sided ideals in $C^*(G;\Sigma)$ and $C_0(G^0)$, respectively, 
such that for every irreducible representation $\pi:C^*(G;\Sigma)\rightarrow B(H_\pi)$ we have
$\mathrm{Res}_M(\ker(\pi))=\ker(\overline{\pi}\circ M)$, where $\overline{\pi}:\mathcal{M}(C^*(G;\Sigma))\rightarrow B(H_\pi)$ denotes the extension of $\pi$ to the  multiplier algebra. Using this it is routine to show the following result:

\begin{lemma}
For every $x\in G^0$ and every irreducible $\sigma_x$-representation $\pi$ of $G_x^x$ we have
$\mathrm{Res}_M(\ker(\mathrm{Ind}_{G_x^x}^G\ \pi))=\{f\in C_0(G^0)\mid f|_{\overline{Gx}}=0\}$.
\end{lemma}

\begin{prop}
The restriction of the map $\mathrm{Res}_M$ to the set of primitive ideals of $C^*(G;\Sigma)$ gives rise to a continuous map
$$p:\mathrm{Prim}(C^*(G;\Sigma))\rightarrow \mathcal{Q}(G)$$
onto the quasi-orbit space of $G$.
\end{prop}
\begin{proof}
    If $I$ is a primitive ideal, then it is the kernel of some irreducible representation $\rho$ of $C^*(G;\Sigma)$. We know that every such representation is induced from a stabiliser, i.e. there exists an $x\in G^0$ and an irreducible $\sigma_x$-representation $\pi$ of $G_x^x$ such that $\rho=\mathrm{Ind}_{G_x^x}^G \pi$. The previous lemma implies $\mathrm{Res}_M(I)=\mathrm{Res}_M(\ker(\rho))=C_0(G^0\setminus \overline{Gx})$.
    But the subspace of $\mathcal{I}(C_0(G^0))$ of ideals of the form $C_0(G^0\setminus \overline{Gx})$ is homeomorphic to the quasi-orbit space. So the composition of $\mathrm{Res}_M$ (restricted and corestricted to a map $\mathrm{Prim}(C^*(G;\Sigma))\rightarrow \{I\in \mathcal{I}(C_0(G^0))\mid I=I_{\overline{Gx}}\}$) with this homeomorphism is the desired continuous map $p$.
\end{proof}

Now let $X_k:=\lbrace x\in G^0\mid \abs{G_x/G_x^x}=k\rbrace$, and $\widetilde{X}_k$ its image in $G^0/G$. 

\begin{lemma}\label{lemma: locally compact Hausdorff}
Let $G$ be an \'etale groupoid.
Then $\widetilde{X_k}$ and $X_k$ are locally compact Hausdorff spaces in the respective relative topology.
\end{lemma}

\begin{proof}
Let us first show that $X_{\leq k}:=\{x\in G^0\mid \vert G_x/G_x^x\vert\leq k\}$ is closed. To see this we show that $G^0\setminus X_{\leq k}$ is open. Let $x\in G^0$ be a point whose orbit has at least $k+1$ distinct elements. Then there exist $g_1,\ldots, g_k\in G$ such that $x,g_1x,\ldots, g_k x$ are pairwise distinct. We also let $g_0:=x$ to have a coherent notation. For each $0\leq i\leq k$ choose an open bisection $U_i$ around $g_i$. Using that $G^0$ is Hausdorff, we can shrink the $U_i$ if necessary to assume without loss of generality that the ranges $r(U_i)$ are pairwise disjoint. But then $V:=\bigcap_{i=0}^k s(U_i)$ is an open neighbourhood of $x$ in $G^0$. By construction, the orbit of every element in $y\in V$ has at least $k+1$ elements, namely the images of $y$ under the partial homeomorphisms $V\subseteq s(U_i)\rightarrow r(U_i)$ given by the bisections $U_i$. Thus, $V\subseteq G^0\setminus X_{\leq k}$ and hence $G^0\setminus X_{\leq k}$ is open.

Since $X_k=G^0\setminus X_{\leq k-1} \cap X_{\leq k}$, it follows that $X_k$ is open in $X_{\leq k}$. We have shown that $X_k$ is locally closed in $G^0$ and hence locally compact.

Since $X_{\leq k}$ and $X_k$ are both $G$-invariant subspaces of $G^0$ the same conclusions easily follow for $\tilde{X}_{\leq k}$ and $\tilde{X}_k$ (note that $G^0/G$ is locally compact since the quotient map is open). 

    Finally, we show that $\tilde{X}_k$ is Hausdorff in the relative topology. So suppose $Gx\neq Gy$. Then in fact all the points in $Gx$ are distinct from all the points in $Gy$. Let us write $Gx=\{x_1,\ldots, x_k\}$ and $Gy=\{y_1,\ldots, y_k\}$. Since all these points are distinct we can find open neighbourhoods $U_i$ of $x_i$ and $V_i$ of $y_i$ such that the $U_i$ and $V_j$ are all pairwise disjoint. Let $U:=\bigcap_{i=1}^k GU_i$ and $V:=\bigcap_{i=1}^k GV_i$. Then $U$ and $V$ are open, $G$-invariant, and we have $Gx\subseteq U$ and $Gy\subseteq V$. Moreover, $U\cap X_k\subseteq \bigcup_{i=1}^k U_i$ and $V\cap X_k\subseteq \bigcup_{i=1}^k V_i$. But this implies $U\cap V\cap X_k=\emptyset$. It follows that $\pi(U)\cap \pi(V)=\emptyset$.
\end{proof}

Next, we restrict our attention to those primitive ideals that arise as kernels of irreducible representations of dimension $k\in \mathbb N$.

\begin{prop}\label{prop:kernels of finite dimensional reps}
The canonical map
$$p_k:\mathrm{Prim}_k(C^*(G;\Sigma))\rightarrow \bigcup_{n\mid k}\widetilde{X}_n$$
sending a representation to the orbit on which it lives is continuous. Moreover, if $x\in X_n$ for $n\leq k$ such that $n$ divides $k$, then we obtain a homeomorphism $$p_k^{-1}(Gx)\cong \mathrm{Prim}_{\frac{k}{n}}(C^*(G_x^x;\sigma_x)).$$
\end{prop}
\begin{proof}
Let $x\in G^0$ such that $\vert Gx\vert=n$. Then $Gx$ is a closed $G$-invariant subset of $G^0$ and hence we have a canonical quotient map
$q:C^*(G;\Sigma)\rightarrow C^*(G|_{Gx};\Sigma|_{Gx})$. Let $I=\ker q$.
Then a classical result of Kaplansky (see \cite[Proposition~3.2.1]{D77}) provides a homeomorphism 
$$\Phi:\mathrm{Prim}_I(C^*(G;\Sigma))\rightarrow \mathrm{Prim}(C^*(G|_{Gx};\Sigma|_{Gx})),$$
where $\mathrm{Prim}_I(C^*(G;\Sigma))$ is the set of two-sided primitive ideals in $C^*(G;\Sigma)$ containing $I$.
Let further $\mathsf{Z}$ be the Hilbert module implementing the Morita equivalence between $C^*(G|_{Gx};\Sigma|_{Gx})$ and $C^*(G_x^x;\sigma_x)$. The Rieffel correspondence tells us that $\mathsf{Z}$ induces a homeomoprhism between the corresponding primitive ideal spaces \cite[Corollary~3.33]{RW98} that we will denote by $\mathrm{Ind}_Z$. It follows that the composition $$\mathrm{Ind}_Z\circ\Phi:\mathrm{Prim}_I(C^*(G;\Sigma))\rightarrow \mathrm{Prim}(C^*(G_x^x;\sigma_x))$$ is a homeomorphism.
The inverse is given by $q_*\circ \mathrm{Ind}_Z=\mathrm{Ind}_X$.
We note that $p_k^{-1}(Gx)=\mathrm{Prim}_I(C^*(G;\Sigma))\cap \mathrm{Prim}_k(C^*(G;\Sigma))$. Using Lemma \ref{Lemma:DimInd} we see that the image of $\mathrm{Prim}_I(C^*(G;\Sigma))\cap \mathrm{Prim}_k(C^*(G;\Sigma))$ under $\mathrm{Ind}_Z\circ \Phi$ coincides with $\mathrm{Prim}_{\frac{k}{n}}(C^*(G_x^x;\sigma_x))$
which finishes our proof.
\end{proof}

We can now collect our findings from this section to prove:
\begin{thm}\label{Thm:DimShmg}
Let $G$ be a second-countable \'etale groupoid and $\Sigma$ be a twist over $G$ such that
\begin{equation}\label{eq:cond for subhomogeneity}
        \sup_{x\in G^{(0)}}\sup_{\pi\in \widehat{(G_x^x,\sigma_x)}} \abs{G_x/G_x^x}\cdot \dim(H_\pi)<\infty.
\end{equation}

\noindent Then the nuclear dimension of $C^*(G;\Sigma)$ can be estimated as follows:
$$\dim^{+1}_{\mathrm{nuc}}(C^*(G;\Sigma))\leq \dim^{+1}(G^0)\cdot \sup_{x\in G^0} \dim_{\mathrm{nuc}}^{+1}(C^*(G_x^x;\sigma_x)), \hbox{ and}$$
$$\sup_{x\in G^0} \dim_{\mathrm{nuc}}(C^*(G_x^x;\sigma_x))\leq\dim_{\mathrm{nuc}}(C^*(G;\Sigma)).$$
\end{thm}
\begin{proof}
We may assume that 
$\sup_x \dim_{\mathrm{nuc}}(C^*(G_x^x;\sigma_x))<\infty$ and $\dim(G^0)<\infty$ as otherwise there is nothing to show. The assumption in line (\ref{eq:cond for subhomogeneity}) implies that $C^*(G;\Sigma)$ is subhomogeneous by Proposition \ref{prop:sh}. Hence we are in a position to apply Winter's theorem \ref{Thm:Winter} and the remaining task is to estimate the dimensions of the spaces $\mathrm{Prim}_k(C^*(G;\Sigma))$ for all $k\in \mathbb N$.
Since every $k$-dimensional irreducible representation of $C^*(G;\Sigma)$ is supported on some finite orbit, we have a decomposition
$\mathrm{Prim}_k(C^*(G;\Sigma))=\bigsqcup_{n\in \mathbb N} p_k^{-1}(\tilde{X}_n)$.
It follows that
$\dim \mathrm{Prim}_k(C^*(G;\Sigma))=\max_{n\mid k} \dim(p_k^{-1}(\widetilde{X}_n))$.
{\color{blue} The} restriction of $p_k$ to a map
$$p_k^{-1}(\widetilde{X}_n)\rightarrow \widetilde{X}_n$$
is a continuous map between second countable, locally compact Hausdorff spaces (see Lemma \ref{lemma: locally compact Hausdorff}). {\color{blue} Via this map, we can view $C_0(p_k^{-1}(\widetilde{X}_n))$ as a $C_0(\widetilde{X}_n)$-algebra and hence} we can apply \cite[Lemma~3.3]{MR3558205} and Proposition \ref{prop:kernels of finite dimensional reps} to get
\begin{align*}
\dim^{+1}(p_k^{-1}(\widetilde{X}_n))&\leq \dim^{+1}(\widetilde{X}_n)\cdot\sup_{x\in X_n}\dim^{+1}(p_k^{-1}(Gx))\\
&=\dim^{+1}(\widetilde{X}_n)\cdot\sup_{x\in X_n}\dim^{+1}(\mathrm{Prim}_{\frac{k}{n}}(C^*(G_x^x;\sigma_x))).
\end{align*}
Since the inverse image of every point under the quotient map $X_n\rightarrow \widetilde{X}_n$ is finite, and using the definition of these spaces, we have 
$\dim(\widetilde{X}_n)=\dim(X_n)\leq \dim(G^0)$.
Moreover, the algebra $C^*(G_x^x;\sigma_x)$ is separable and subhomogeneous for all $x\in G^0$. Consequently, another application of Winter's theorem \ref{Thm:Winter} implies
$$\dim(\mathrm{Prim}_{\frac{k}{n}}(C^*(G_x^x;\sigma_x)))\leq \dim_{\mathrm{nuc}}(C^*(G_x^x;\sigma_x)),$$
which concludes the proof of the first inequality.

For the proof of the second inequality, recall that the general assumption that orbits are finite implies that the orbit space $G^0/G$ is $T_1$. Hence $C^*(G_x^x;\sigma_x)$ is stably isomorphic to a quotient of $C^*(G;\Sigma)$. The claimed inequality then follows from the known permanence properties of $\dim_{\mathrm{nuc}}$ in \cite{WZ10}.
\end{proof}

\begin{bem}\label{Rem} Suppose we are in the situation of Theorem \ref{Thm:DimShmg}.
If the twist $\Sigma$ is trivial, or $G$ is principal, we also have the estimate
    $$\dim(G^0)\leq \dim_{\mathrm{nuc}}(C^*(G;\Sigma)).$$
    Indeed, in either case inducing the trivial representation of $G_x^x$ gives rise to a continuous map $l:G^0\rightarrow \mathrm{Prim}(C^*(G;\Sigma))$ (see \cite{C07,CaH12}). The preimage of the representation $\mathrm{Ind}_{G_x^x}^G 1$ is precisely the orbit $Gx$, which is finite, and hence zero-dimensional. It follows that for each $k\in \mathbb N$ we have
    \begin{align*}
         \dim^{+1}(X_k)&\leq \dim^{+1}(\mathrm{Prim}_k(C^*(G;\Sigma)))\sup_{I\in \mathrm{Prim}_k(C^*(G;\Sigma))} \dim^{+1}(l^{-1}(I))\\
         &=\dim_{\mathrm{nuc}}^{+1}(C^*(G;\Sigma)).
    \end{align*}
   
    As we have seen in the proof of Lemma \ref{lemma: locally compact Hausdorff}, the space $X_{\leq k}$ is a finite union of the (relatively) open sets $X_n$, $n\leq k$ and hence $\dim(X_{\leq k})\leq \dim_{\mathrm{nuc}}(C^*(G;\Sigma))$ by the sum theorem for open sets \cite[3.5.10]{Pears}. Each of the sets $X_{\leq k}$ in turn is closed in $G^0$ and $G^0=\bigcup_{k\in \mathbb N} X_k$, and hence $\dim(G^0)\leq \dim_{\mathrm{nuc}}(C^*(G;\Sigma))$ by the countable sum theorem \cite[2.2.5]{Pears}.
\end{bem}

 Observe, that under the assumptions of Theorem \ref{Thm:DimShmg}, the twisted $\mathrm{C}^*$-algebras of the isotropy groups $C^*(G_x^x;\sigma_x)$ are all subhomogeneous. In particular, $G_x^x$ is a countable, discrete, virtually abelian group and $\sigma_x$ is type I by \cite[Theorem 2]{K83}. As the next proposition demonstrates, we can use this to replace $\dim_{\mathrm{nuc}}(C^*(G_x^x;\sigma_x))$ by the asymptotic dimension of $G_x^x$, which completes the proof of Theorem \ref{Thm:A}.

\begin{prop}\label{nucdim virab}
Let $\Gamma$ be a countable, discrete, virtually abelian group with a normalized 2-cocycle $\sigma$ with values in $\mathbb{T}$. If $\sigma$ is type I, then we have

$$
\mathrm{dim}_{\mathrm{nuc}}(C^*(\Gamma;\sigma))\leq \mathrm{asdim}(\Gamma).
$$
\end{prop}
\begin{proof}
As $\sigma$ is type I, there is an abelian subgroup $\Lambda\subseteq \Gamma$ of finite index such that $\sigma$ is trivial on $\Lambda$ (see \cite[Theorem 2]{K83}). Since the normalizer subgroup $N_\Gamma(\Lambda)$ contains $\Lambda$, there are only finite number of conjugates $\{g_i\Lambda g_i^{-1}\}_{i=1,\ldots, k}$. Now let $N:=\cap_{i=1}^kg_i\Lambda g_i^{-1}$ which is a subgroup of $\Lambda$. It is not hard to see that $N$ is an abelian,  normal, finite-index subgroup of $\Gamma$ such that $\sigma$ is trivial on $N$.


To show $\mathrm{dim}_{\mathrm{nuc}}(C^*(\Gamma;\sigma))\leq \mathrm{asdim}(\Gamma)$, we write $C^*(\Gamma;\sigma)\cong C^*(N)\rtimes_{\alpha,\omega}G/N$ as a twisted crossed product using \cite[Proposition~4.1]{MR1002543}. Since the quotient map $q:\widehat{N}\rightarrow Z:=\widehat{N}/(G/N)$ is continuous and open, $C^*(N)$ is a continuous field over $Z$. As $G/N$ acts trivially on $Z$, $q$ is actually $G/N$-equivariant. Therefore, $C^*(\Gamma;\sigma)$ is also a continuous field over $Z$ with fiber algebra $C^*(\Gamma;\sigma)_{[\gamma]}\cong C^*(N)_{[\gamma]}\rtimes_{\widetilde{\alpha},\widetilde{\omega}}G/N$, where $[\gamma]\in Z$ (see e.g. \cite[Theorem~5.1 and Corollary~5.3]{MR1414338}). As $C^*(N)_{[\gamma]}\cong \oplus_{\eta\in [\gamma]} \mathbb{C}$ by \cite[Proposition~3.2]{MR3917214} and $G/N$ is a finite group, we conclude that $C^*(\Gamma;\sigma)_{[\gamma]}$ is a finite-dimensional $C^*$-algebra. Finally, it follows from \cite[Lemma~3.3]{MR3558205} and \cite[Corollary~5]{MR3022735} that
\begin{align*}
\dim^{+1}_{\mathrm{nuc}}(C^*(\Gamma;\sigma))&\leq \dim^{+1}(Z)\cdot \sup_{[\gamma]\in Z} \dim^{+1}_{\mathrm{nuc}}(C^*(\Gamma;\sigma)_{[\gamma]})\\
&=\dim^{+1}(\widehat{N})=\mathrm{asdim}^{+1}(N) =\mathrm{asdim}^{+1}(\Gamma).
\end{align*}
\end{proof}
\begin{quest}
Do we actually have $dim_{nuc}(C^*(\Gamma;\sigma))= asdim(\Gamma)$ in Proposition~\ref{nucdim virab}?   
\end{quest}


\section{Applications to nuclear dimension of non-principal twisted groupoid $C^*$-algebras}
\label{Sec:dad}
In this section we apply our results from Section \ref{Sec:Subhom} to show that the principality assumptions in \cite[Theorem~8.6]{GWY17} and its recent generalisation in the twisted case (see \cite[Theorem~4.1]{CDGHV22}) are redundant.
Let us first recall the definition of dynamic asymptotic dimension.

\begin{defi}{\cite[Definition~5.1]{GWY17}} Let $G$ be a locally compact, Hausdorff and étale groupoid and $d\in \mathbb{N}_0$. Then $G$ has \textbf{dynamic asymptotic dimension} at most $d$, if for every open and precompact subset $K\subseteq G$, there exists a cover of $s(K)\cup
r(K)$ by $d+1$ open subsets $U_0,\ldots,U_d$ of $G^0$ such that for each $0\leq i\leq d$, the open subgroupoid $\langle K\cap G|_{U_i}\rangle$ is precompact in $G$. We write $\mathrm{dad}(G)$ for the minimal $d\in \mathbb N_0$ satisfying the above and call it the dynamic asymptotic dimension of $G$.
\end{defi}

The reader can find further information on this notion and many examples in \cite{GWY17,CDGHV22,B23}. We will go straight to our main application:

\begin{thm}\label{nucdim}
Let $G$ be a second-countable, locally compact, Hausdorff étale groupoid and let $j:\Sigma\to G$ be a twist over $G$. Then
$$\dim^{+1}_{\mathrm{nuc}}(C_r^*(G;\Sigma))\leq \mathrm{dad}^{+1}(G)\cdot \dim^{+1}(G^0).$$
\end{thm}
\begin{proof}
The main idea {\color{red} is to follow the the proof of \cite[Theorem~4.1]{CDGHV22} and replace the invocation of \cite[Proposition~ 4.3]{CDGHV22} therein by our Theorem~\ref{Thm:DimShmg}.}

We may suppose without loss of generality that $d:= \mathrm{dad}(G)<\infty$ and $N:=\dim(G^0)<\infty$ as otherwise there is nothing to show.

It is furthermore sufficient to consider the case where $G^0$ is compact. Indeed, if $G^0$ is not compact, we consider the Alexandrov groupoid $\tilde{G}$ of $G$ and the Alexandrov twist $\tilde{\Sigma}$ over $\tilde{G}$. It follows from \cite[Proposition~3.13, Lemma~2.6 and Lemma~3.8]{CDGHV22} that $\mathrm{dad}(G)=\mathrm{dad}(\tilde{G})$, $\dim(\tilde{G}^0)=\dim(G^0)$ and $C_r^*(\tilde{G};\tilde{\Sigma})$ is the minimal unitization of $C_r^*(G;\Sigma)$. Therefore, $C_r^*(\tilde{G};\tilde{\Sigma})$ and $C_r^*(G;\Sigma)$ have the same nuclear dimension (see \cite[Remark~2.11]{WZ10}).  

Let us now assume that the unit space $G^0$ is actually compact.
Let $\mathcal{F}$ be a finite subset of $C_c(G;\Sigma)\backslash \{0\}$ and let $\epsilon>0$. There exists a compact subset $K$ of $\Sigma$ such that $f\in \mathcal{F}$ implies $\text{supp} f\subseteq K$. Since both $K^{-1}$ and $j^{-1}(j(K))$ are compact sets (see \cite[Lemma~2.2]{CDGHV22}), we may assume that $K=K^{-1}$ and that $K=j^{-1}(j(K))$. Since $G^0=\Sigma^0$ is compact and open in $\Sigma$, we may also assume that $G^0\subseteq K$ and hence that $1\in \mathcal{F}$.

Let $V\subseteq G$ be an open and precompact neighborhood of $j(K)$, and let $\delta$ be as in \cite[Lemma~4.5 (2)]{CDGHV22} for $\frac{\epsilon}{(d+1)\max_{f\in \mathcal{F}}||f||_{C^*_r(G;\Sigma)}} $, $K$ and $V$. Since $G$ has dynamic asymptotic dimension $d$, applying \cite[Proposition~7.1]{GWY17} to $\delta$ and the precompact, open subset $V$ of $G$ gives
\begin{itemize}
\item[(1)] open subsets $U_0,\ldots,U_d$ covering $G^0=r_G(V)\cup s_G(V)$ such that the subgroupoids $H_i$ generated by $\{\gamma\in V:s_G(\gamma),r_G(\gamma)\in U_i\}$ are open and precompact in $G$ for $0\leq i\leq d$;
\item[(2)] continuous and compactly supported functions $h_i:G^0\rightarrow [0,1]$ with support in $U_i$ such that for $x\in G^0$ we have $\sum_{i=0}^dh_i(x)^2=1$, and for $\gamma\in V$ and $0\leq i\leq d$ we have
$$
\sup_{\gamma\in V}|h_i(s_G(\gamma))-h_i(r_G(\gamma))|<\delta.
$$
\end{itemize}

Let $0\leq i\leq d$. Because $H_i$ is open, $j^{-1}(H_i)$ is a twist over $H_i$ by \cite[Lemma~2.5]{CDGHV22}. Since each $H_i$ is a second-countable \'etale open subgroupoid which is precompact in $G$, we can cover $H_i$ by finitely many open bisections $U_1,\ldots, U_M$. It follows that $(H_i)_x$ and $(H_i)_x^x$ have at most $M$ elements for all $x\in H_i^0$. By the twisted version of the Peter-Weyl Theorem (see for example \cite[Corollary~7.15]{CSST22}), it follows that each $H_i$ satisfies the assumptions of Theorem~\ref{Thm:DimShmg} and hence that $C^*(H_i;j^{-1}(H_i))$ has nuclear dimension at most the covering dimension of $H_i^0$, which is bounded by $N$. {\color{red} As each $H_i$ is an amenable groupoid by \cite[Lemma~A.10]{GWY16}, its full and reduced twisted groupoid $C^*$-algebras coincide.} We can thus follow the argument in \cite[Theorem~4.1]{CDGHV22} again to obtain the estimate $\dim_{\mathrm{nuc}}(C_r^*(G;\Sigma))\leq (N+1)(d+1)-1$ from \cite[Proposition~4.2]{CDGHV22}. As the remaining {\color{red} details are} identical to the proof in \cite[Theorem~4.1]{CDGHV22}, we omit them.
\end{proof}

The following result provides us with a class of not necessarily free actions to which one can apply Theorem \ref{nucdim}.

\begin{thm}\label{dad D_inf}
The dynamic asymptotic dimension of any minimal action $D_\infty \curvearrowright X$ of the infinite dihedral group $D_\infty$ on an infinite compact Hausdorff space $X$ is one.

\smallskip

If we further assume that $X$ is second-countable and has finite covering dimension, then $C(X)\rtimes D_\infty$ is classifiable
by its Elliott invariant and has nuclear dimension at most one.
\end{thm}
\begin{proof}

By \cite[Theorem~2.2 and Definition~3.1]{MR4277767}, it suffices to show that any minimal action $D_\infty \curvearrowright X$ has the marker property. More precisely, we aim to show that for any finite subset $F\subseteq D_\infty$ there exists a non-empty open subset $U\subseteq X$ such that
\begin{itemize}
\item[(1)] $gU\cap g'U=\emptyset$ for any two distinct $g,g'\in F$;
\item[(2)] $X=D_\infty U$.
\end{itemize}
As $D_\infty U$ is a non-empty open invariant subset of $X$ and the action is minimal, (2) holds automatically.
To verify (1), we write 
$$
D_\infty=\mathbb{Z}\rtimes \mathbb{Z}/2\mathbb{Z}=\langle s,t\ |\  t^2=1, tsts=1 \rangle=\{s^n,s^nt\ |\ n\in \mathbb{Z}\},
$$
where $\mathbb{Z}=\langle s \rangle \lhd D_\infty$. For any $x\in X$, we consider the stabilizer $Stab(x):=\{g\in D_\infty\ | \ gx=x\}$, which has infinite index in $D_\infty$. Indeed, $[D_\infty: Stab(x)]=|D_\infty x|=\infty$ as $D_\infty x$ is dense in the infinite space $X$. Hence,
$$
[\mathbb{Z}: \mathbb{Z}\cap Stab(x)]=\frac{[D_\infty: \mathbb{Z}\cap Stab(x)]}{[D_\infty: \mathbb{Z}]}\geq \frac{[D_\infty:  Stab(x)]}{[D_\infty: \mathbb{Z}]}=\infty.
$$
As every non-trivial subgroup of $\mathbb{Z}$ has the form $n\mathbb{Z}$, we obtain $\mathbb{Z}\cap Stab(x)=\{e\}$. As $D_\infty=\{s^n,s^nt\ |\ n\in \mathbb{Z}\}$, $ Stab(x)\subseteq \{e, s^nt\ |\ n\in \mathbb{Z}\}$. Hence, it follows easily that either $Stab(x)=\{e, s^nt\}$ for some $n\in \mathbb{Z}$ or $Stab(x)=\{e\}$.

Case I: If $Stab(x)=\{e\}$, then $gx\neq g'x$ for all $g\neq g'$ in $F$. As $X$ is Hausdorff, we may find open sets $U_g\subseteq X$ for each $g\in F$ such that $gx\in U_g$ and $U_g\cap U_{g'}=\emptyset$ for any two distinct $g,g'\in F$. Then the non-empty open subset $U:=\cap_{g\in F}g^{-1}U_g$ does the job.

Case II: If $Stab(x)=\{e, s^nt\}$ for some $n\in \mathbb{Z}$, there exists $k\in \mathbb{Z}$ such that $Stab(s^kx)\cap (F^{-1}F\backslash \{e\})=\emptyset$. To see this, we write $F=\{s^i, s^jt\ |\ i\in I, j\in J\}$ for some finite subsets $I,J\subseteq \mathbb{Z}$. Using the relations $t^2=e$ and $tst^{-1}=s^{-1}$, we see that
$$
Stab(s^kx)=s^kStab(x)s^{-k}=s^k\{e, s^nt\}s^{-k}=\{e, s^{k}s^nts^{-k}\}=\{e, s^{2k+n}t\}.
$$
It follows that
\begin{align*}
Stab(s^kx)\cap (F^{-1}F\backslash \{e\})\neq \emptyset \Leftrightarrow&\  s^{2k+n}t\in F^{-1}F\backslash \{e\}\\
 \Leftrightarrow&\  s^{2k+n}t=s^{-i+j'}t\ \text{or} \ s^{j-i'}t\ \text{for some $i,i'\in I, j,j'\in J$}\\
 \Leftrightarrow&\  2k+n=j'-i \ \text{or} \ j-i' \ \text{for some $i,i'\in I, j,j'\in J$}.
\end{align*}
If we choose $k$ large enough such that $|2k+n|>\max_{i\in I, j\in J}|i-j|$, then $Stab(s^kx)\cap (F^{-1}F\backslash \{e\})=\emptyset$. In particular $g(s^kx)\ne g'(s^kx)$ for all $g\neq g'\in F$. Applying the same argument as in Case I to $s^kx$ instead of $x$, we obtain the desired non-empty open subset $U$. Hence, the action has the marker property.
\\

If we also assume that $X$ is second-countable and has finite covering dimension, then $C(X)\rtimes D_\infty$ has finite nuclear dimension by Theorem~\ref{nucdim}. By \cite[Proposition~2.6]{MR4584676} and \cite{MR1703305}, $C(X)\rtimes D_\infty$ is a simple $C^*$-algebra in the UCT class. Therefore, it is classifiable. By \cite[Theorem~A and Theorem~B]{MR4228503} and \cite[Theorem~A]{MR3583354}, $C(X)\rtimes D_\infty$ has decomposition rank (hence also nuclear dimension) at most one.
\end{proof}

\begin{ex}\label{ex:Edu}
We refer the reader to \cite{MR4130816} for certain non-free $D_\infty$-odometers, which were shown to be counterexamples to the
HK conjecture. Since they are all Cantor minimal $D_\infty$-systems, it follows from Theorem~\ref{dad D_inf} that their dynamic asymptotic dimension is one and the nuclear dimension of the associated crossed product is bounded by one. It follows in particular, that these crossed products are classifiable.
\end{ex}
The following corollary complements \cite[Corollary~2.5]{MR4277767}:
\begin{cor}\label{vircyc nuc}
Let $X$ be an infinite compact Hausdorff space. If $\Gamma$ is a virtually cyclic group acting minimally (not necessarily topologically free) on $X$, then the dynamic asymptotic dimension of $\Gamma  \curvearrowright X$ is one and
\begin{align}\label{inqe}
\dim_{\mathrm{nuc}}(C(X)\rtimes_r \Gamma)\leq 2\cdot \dim(X)+1.
\end{align}
\end{cor}
\begin{proof}
Since $X$ is infinite and the action is minimal, the group $\Gamma$ must be infinite. As $\Gamma$ is an infinite virtually cyclic group, it has a finite normal subgroup $N\subseteq \Gamma$ such that $\Gamma/N$ is either $\mathbb{Z}$ or $D_\infty$. It is easy to deduce that the dynamic asymptotic dimension of $\Gamma  \curvearrowright X$ is bounded by the dynamic asymptotic dimension of the minimal action $\Gamma /N \curvearrowright X/N$, which is equal to one by \cite[Theorem~3.1]{GWY17} for $\Gamma/N=\mathbb{Z}$ and Theorem~\ref{dad D_inf} for $\Gamma/N=D_\infty$. The dynamic asymptotic dimension equals 0 only for actions of locally finite groups, and $\Gamma$ contains an infinite-order element. Thus, we have completed the proof of the first statement.

If $X$ is second-countable, it follows directly from Theorem~\ref{nucdim} that the inequality (\ref{inqe}) holds. If $X$ is not second-countable, the inequality (\ref{inqe}) follows from the second-countable case of $X$ via a direct limit argument (see \cite[ Lemma~1.3 ]{MR3558205} and \cite[Proposition~2.3~(iii)]{WZ10}).
\end{proof}

\begin{bem}
In Corollary~\ref{vircyc nuc}, the minimal action may \emph{not} be topologically free and the $C^*$-algebra may \emph{not} be simple. Indeed, if we consider any minimal action of $\mathbb{Z}$ on $X$ and any non-trivial finite group $F$. Then $\mathbb{Z}\times F$ acts minimally on $X$ when $F$ acts trivially on $X$. However, this minimal action of $\mathbb{Z}\times F$ is not topologically free because $X^g=X$ for all non-trivial $g\in F$.
\end{bem}

We end the paper by providing some further positive evidence towards the following open question:

\begin{quest}\label{Z-stablequesti}
Does every separable nuclear $\mathcal{Z}$-stable $C^*$-algebra $A$ have finite nuclear dimension?
\end{quest}
\begin{prop}\label{parcon}
Let $G$ be a second-countable, locally compact, Hausdorff and étale groupoid and let $\Sigma$ be a twist over $G$. Suppose that $G$ has dynamic asymptotic dimension $d$. Then the nuclear dimension of $C_r^*(G;\Sigma)\otimes \mathcal{Z}$ is at most $3d+2$, where $\mathcal{Z}$ is the Jiang-Su algebra.
\end{prop}
\begin{proof}
The proof is a slight variant of the one for Theorem~\ref{nucdim}. Indeed, let $\{H_i\}_{0\leq i\leq d}$ be the second-countable étale open precompact subgroupoids of $G$ as constructed in the proof of Theorem~\ref{nucdim}. By Proposition~\ref{prop:sh}, the twisted $C^*$-algebras $C_r^*(H_i;j^{-1}(H_i))$ are all subhomogenuous. As $\mathcal{Z}$ is locally subhomogeneous, it follows from \cite[Theorem A]{MR4116643} that $C_r^*(\pi^{-1}(H_i);H_i)\otimes \mathcal{Z}$ has nuclear dimension at most 2. Therefore, $\dim_{\mathrm{nuc}}(C_r^*(G;\Sigma)\otimes \mathcal{Z})\leq (2+1)(d+1)-1=3d+2$ by following the argument in the proof of Theorem~\ref{nucdim}. 
\end{proof}

\begin{bem}
So far we have an affirmative answer to Question~\ref{Z-stablequesti} in the following three cases:
\begin{itemize}
\item if $A$ is simple (see \cite{MR4228503,MR4310098});

\item if $A$ is traceless (see \cite{MR4392219, MR2106263});

\item if $A$ is a twisted étale groupoid $C^*$-algebra with finite dynamic asymptotic dimension (see Proposition~\ref{parcon}). 
\end{itemize} 
\end{bem}
\textbf{Acknowledgement}: We wish to thank Jamie Gabe, Yongle Jiang and Eduardo Scarparo for helpful discussions on Question~\ref{Z-stablequesti} and minimal $D_\infty$-actions.
\bibliographystyle{plain}
\bibliography{references}
\end{document}